\documentclass[12pt]{amsart}
\usepackage{amsmath,amsthm,amssymb,amsfonts}
\usepackage[nobysame,abbrev,alphabetic]{amsrefs}
\usepackage{color}

\numberwithin{equation}{section}

\DeclareFontFamily{OT1}{rsfs}{}
\DeclareFontShape{OT1}{rsfs}{n}{it}{<-> rsfs10}{}
\DeclareMathAlphabet{\mathscr}{OT1}{rsfs}{n}{it}

\theoremstyle{definition}

\newcommand{\IGNORE}[1]{}

\newtheorem{theorem}{Theorem}[section]
\newtheorem{proposition}[theorem]{Proposition}

\newtheorem{lemma}[theorem]{Lemma}

\theoremstyle{definition}
\newtheorem{definition}[theorem]{Definition}
\newtheorem{remark}[theorem]{Remark}

\newcommand{\eps}{\epsilon}

\newcommand{\dstyle}{\displaystyle}
\newcommand{\hide}[1]{}

\begin{document}

\title[Finite total $Q$-curvature]{On locally conformally flat manifolds with finite total $Q$-curvature}

\author[Zhiqin Lu]{Zhiqin Lu}
\address{Zhiqin Lu, Department of Mathematics, University of California, Irvine,
410D Rowland Hall, Irvine, CA 92697}
\address{ email: zlu@uci.edu}
\thanks{The first author is partially supported
by NSF grant DMS-1510232, and  the econd author is partially supported
by NSF grant DMS-1547878}

\author[Yi Wang]{Yi Wang}
\address{Yi Wang, Department of Mathematics, Johns Hopkins University, 404 Krieger Hall, 3400 N. Charles Street, Baltimore, MD 21218,}
\address{ email: ywang@math.jhu.edu}
\setcounter{page}{1}


\begin{abstract}In this paper, we focus our study on the ends of a locally conformally flat complete manifold with finite total $Q$-curvature.
We prove that for such a manifold, the integral of the $Q$-curvature equals an integral multiple of a dimensional constant $c_n$, where $c_n$ is the integral of the $Q$-curvature on the unit $n$-sphere.
It provides further evidence that the $Q$-curvature on a locally conformally flat manifold controls geometry as the Gaussian curvature does in two dimension.
\end{abstract}
\maketitle

\section{Introduction}
The $Q$-curvature arises naturally as a conformal invariant associated to
the Paneitz operator. When $n=4$, the Paneitz operator is defined as:
$$P_g=\Delta^2+\delta(\frac{2}{3}Rg-2 Ric)d,$$
where $\delta$ is the divergence, $d$ is the differential, $R$ is the scalar curvature of $g$, and $Ric$
is the Ricci curvature tensor. The Branson's $Q$-curvature \cite{Branson} is defined as
$$Q_g=\frac{1}{12}\left\{-\Delta R +\frac{1}{4}R^2 -3|E|^2 ,\right\}
$$
where $E$ is the traceless part of $Ric$, and $|\cdot|$ is taken with respect to the metric $g$.
Under the conformal change $g_{w}=e^{2w}g_0$, the Paneitz operator transforms by $P_{g_w}=e^{-4w}P_{g_0}$,
and
$Q_{g_w}$ satisfies the fourth order equation
$$P_{g_{0}}w+2Q_{g_0}=2Q_{g_{w}}e^{4w}.$$
This is analogous to the transformation law satisfied by the Laplacian operator $-\Delta_g$ and the Gaussian curvature $K_g$ on surfaces,
$$-\Delta_{g_0}w+K_{g_0}=K_{g_w}e^{2w}.$$

The invariance of $Q$-curvature in dimension $4$ is due to the Gauss-Bonnet-Chern formula for a closed manifold $M$:
\begin{equation}\label{GBCEq}\chi(M)=\dstyle\frac{1}{4\pi^2} \int_{M}\left(\frac{|W|^2}{8}+Q_g\right) dv_g,\end{equation}
where $W$ denotes the Weyl tensor. Chang-Qing-Yang proved in \cite{CQY1} the following theorem.

Let $(M^4,g)=(\mathbb{R}^4, e^{2w} |dx|^2)$ be a noncompact complete conformally flat manifold
with finite total $Q$-curvature, i.e. $\int_{M^4}|Q_g|d v_g<\infty$. If the metric is normal, i.e.
\begin{equation}\label{wFlat}
w(x)=\frac{1}{4\pi^2}\int_{\mathbb{R}^4}\log\frac{|y|}{|x-y|}Q_g(y)e^{4w(y)} dy + C,
\end{equation} or if the scalar curvature $R_g$ is nonnegative at infinity, then
\begin{equation}\label{1.3}
\dstyle \frac{1}{4\pi^2}\int_{M^4} Q_g dv_g\leq  \chi(\mathbb{R}^4)=1,
\end{equation}
and \begin{equation}\label{1.4}
\displaystyle \chi(\mathbb{R}^4)-\frac{1}{4\pi^2}\int_{\mathbb{R}^4} Q_g dv_g= \sum_{j=1}^{k}\lim_{r\rightarrow
\infty }\frac{vol_g(\partial B_{j}(r))^{4/3}}{4(2\pi^2)^{1/3}vol_g(B_{j}(r))},
\end{equation}
\noindent where $B_{j}(r)$ denotes the Euclidean ball with radius $r$ at the $j$-th end.

The theorem of Chang-Qing-Yang  asserts that for $4$-manifolds (in fact, their theorem is valid for all even dimensions) which is conformal to the Euclidean space, the integral of the $Q$-curvature controls the asymptotic isoperimetric ratio at the end of this complete manifold.
This is analogous to the two-dimensional result by
Cohn-Vossen \cite{Cohn-Vossen}, 
 who studied the Gauss-Bonnet integral for a noncompact complete surface $M^2$ with analytic
metric, and showed that if the manifold has finite total Gaussian curvature, then
\begin{equation}\label{1.1}
\dstyle \frac{1}{2\pi}\int_{M} K_g dv_g\leq  \chi(M),
\end{equation}
where $\chi(M)$ is the Euler characteristic of $M$. Later, Huber \cite{Huber} and Hartman \cite{hartman} extended this inequality to
metrics with much weaker regularity. Huber also proved that such a surface
$M^2$ is conformally equivalent to a closed surface with finitely many points removed. The difference of the two sides in inequality (\ref{1.1}) encodes the asymptotic behavior of the manifold at its ends.  The precise geometric interpretation has been given by Finn \cite{Finn} as follows. Suppose a noncompact complete surface has absolutely integrable Gaussian curvature. Then one may
represent each end conformally as $\mathbb{R}^2 \setminus K$ for some compact set $K$. Define the asymptotic
isoperimetric constant of the $j$-th end to be
$$\nu_i=\lim_{r\rightarrow \infty}\frac{L_g^2(\partial B(0,r)\setminus K)}{4\pi A_g(B(0,r)\setminus K)},
$$
where $B(0,r)$ is the Euclidean ball centered at origin with radius $r$, $L$ is the
length of the boundary, and $A$ is the area of the domain. Then
\begin{equation}\label{1.2}
\displaystyle \chi(M)-\frac{1}{2\pi}\int_{M} K_g dv_g= \sum_{j=1}^{N}\nu_{j},
\end{equation}
where $N$ is the number of ends on $M$. This result tells us that the condition of finite total
Gaussian curvature has rigid geometric and analytical consequences.

The results of Chang, Qing and Yang
(\ref {1.3}), (\ref {1.4}) are higher dimensional counterparts of (\ref {1.1}), (\ref {1.2}). Moreover, in \cite{CQY2}, they also generalized these results to locally conformally flat manifolds with certain curvature conditions and obtained the conformal compactification of such manifolds.

In this paper, we aim to continue the study of the integral of $Q$-curvature over complete locally conformally flat manifold.
For a closed locally conformally flat 4-manifold,  \eqref{GBCEq} yields
$$\chi(M)=\frac1{4\pi^2}\int_MQ_gd v_g.$$
For a complete locally conformally flat 4-manifold, the asymptotic behavior near the end is important. In the main result of this paper, we give sufficient conditions to control the asymptotic behavior of the ends, and thus control the integral of $Q$-curvature.

\begin{theorem}\label{main theorem}
Let $(M^4,g)$ be a complete locally conformally flat
manifold with finite total $Q$-curvature and finite number of conformally flat simple ends. 
Suppose on each end the metric is normal, or the scalar curvature is nonnegative at infinity. If $M^4$ is immersed in $\mathbb{R}^{5}$ with
\begin{equation}
\dstyle \int_{M^4} |L|^4dv_g< \infty,
\end{equation}
with $L$ being the second fundamental form,
then
$$\int_{M^4} Q_g dv_g\in8\pi^2 \mathbb{Z}.$$
\end{theorem}
We refer to Definition \ref{normal} for the definition of normal metric, and  referr to Definition \ref{simple ends} for the definition of conformally flat simple end.\\

Theorem \ref{main theorem} is the higher dimensional analog of what is known for the Gaussian curvature on surfaces, which was proved by Chern-Osserman \cite{Chern-Osserman} for minimal surfaces; and was proved by White \cite{White} for general surfaces.\\

\begin{definition} \label{normal}
The metric is  normal on an end $E_j$ of a locally conformally flat manifold if $(E_j,g)=(\mathbb{R}^4\setminus B, e^{2w}|dx|^2)$ and

\begin{equation}\label{w}
w(x)=\frac{1}{4\pi^2}\int_{\mathbb{R}^4\setminus B}\log\frac{|y|}{|x-y|}Q_g(y)e^{4w(y)} dy ,
\end{equation}
where
$B$ is a ball with respect to the Euclidean metric.
\end{definition}

We remark that the existence of  normal metric is a necessary assumption in this theorem. Without such an assumption, there may exist quadratic functions in the kernel of the bi-Laplacian operator $\Delta^2$ (with respect to the flat metric) for which (\ref{1.3}) fails. Note that the assumption of positive scalar curvature at infinity would imply that the metric is normal in dimension 4 (see for example Proposition 1.12 in \cite{CQY2}), therefore for the purposes of this theorem, it can be replaced the condition of a normal metric.

We adopt the definition from \cite{CQY2} of manifolds with conformally flat simple ends.
\begin{definition} \label{simple ends}
Suppose that $(M^n, g)$ is a complete manifold such that
\begin{equation*}
M^n =N^n \bigcup \left\{ \bigcup_{j=1}^N E_j \right\},
\end{equation*}
where $(N^n, g)$ is a compact manifold with boundary
\begin{equation}
\partial  N^n=   \bigcup_{j=1}^N \partial  E_j ,
\end{equation}
and each $E_j$ is a conformally flat simple end of $M^n$; that is
\begin{equation*}
(E_j, g)=  (\mathbb R^n \setminus B, e^{2w} |dx |^ 2).
\end{equation*}
Here $B$ is a ball with respect to the Euclidean metric.

\end{definition}



\begin{remark} Note that the model case of the manifold described in Theorem \ref{main theorem} is a cone. However it should be clear that $Q_ge^{4w}$ could be very close to the distribution:
$$\displaystyle  \sum_{k=2}^{\infty}\frac{1}{k^2} \delta_{k},$$
where $\delta_{k}$ denotes the Dirac mass centered at point $x=(k,0,...,0)$.
Therefore, we cannot expect the metric to be close to a cone metric outside any compact set. Nor can we expect that the Ricci curvature of the manifold is bounded. Thus the method by estimating Ricci curvature lower bound cannot be applied here.
\end{remark}



Theorem \ref{main theorem} is not restricted to $4$-dimension. 
But it is technically more complicated in higher dimensions.

\begin{theorem}\label{mainhigh}
Let $(M^n,g)$ be an even dimensional locally conformally flat complete
manifold with finite total $Q$-curvature and finitely many conformally flat simple ends.
Suppose that on each end, the metric is normal. If $M^n$ is immersed in $\mathbb{R}^{n +1}$ with
\begin{equation}
\dstyle \int_{M^n} |L|^ndv_g< \infty,
\end{equation}
with $L$ being the second fundamental form,
then
$$\int_{M^n} Q_g dv_g\in 2c_n\mathbb{Z},$$
where $c_n= 2 ^ {n -2} (\frac {n -2} 2)! \pi ^ {\frac n 2}$ is the integral of the $Q$-curvature on the standard $n $-hemisphere $\mathbb{S}^n_+$.
\end{theorem}

\begin{definition}
The metric is normal on an end $E_j \subset M ^ n$ of a locally conformally flat manifold  $M^n$ if
$(E_j,g)=(\mathbb{R}^n\setminus  B, e^{2w}|dx|^2)$ and
$$w(x)=\frac{1}{c_n}\int_{\mathbb{R}^n\setminus B}\log\frac{|y|}{|x-y|}P(y)dx+C
$$
for some continuous $L^1(\mathbb{R}^n\setminus B)$ function $P(y)$. The dimensional constant $c_n$ defined in Theorem \ref{mainhigh} is also the constant that appears in the fundamental solution equation $(-\Delta)^{n/2}\log\frac{1}{|x|}=c_n\delta_0(x).$
\end{definition}

\noindent \textbf{Acknowledgments:}
The second author is grateful to Alice Chang and Paul Yang for discussions and interest to this work. She would also like to thank Matt Gursky for interest to the work and suggestions.

\section {On the integral of divergence terms.}

In the following lemma, we show that the integral of $\Delta_g R_g$ vanishes, assuming that there exists a cut-off function $\eta_\rho$ with its Hessian estimates.
However, the existence of such a cut-off function  is in general not true,
since we do not have Ricci curvature
lower bound. Later, in order to remove this assumption, we will adopt an argument making full
use of the special structure of conformally flat ends.

We start with the following general result
\begin{lemma}\label{vanishing0}
Let $M$ be a $4$-dimensional Riemannian manifold.
Assume that the $Q$-curvature is absolutely integrable; the second fundamental form $L$ is in $L^4(M)$; and the metric is  normal on each end $E_j$.
For a fixed point $p$, assume  also that there exists a smooth cut-off function $\eta_\rho$ which is supported on the geodesic ball $B^g(p, 2\rho)$; it is equal to $1$ on $B^g(p, \rho)$; its gradient
is of order $O(1/\rho)$ on $B^g(p, 2\rho)\setminus B^g(p, \rho)$; and its Hessian is of order $O(\frac 1{\rho^2})$ on$B^g(p, 2\rho)\setminus B^g(p, \rho)$.
Then $$\int_{M} \Delta_g R_g dv_g=0.$$
\end{lemma}

\begin{proof}[Proof of Lemma~\ref{vanishing0}]
Since the $Q$-curvature is absolutely integrable, so is $\Delta_gR$.
By using the smooth cut-off function  $\eta_\rho$,  we have
\begin{equation}\label{vanishing0Eq1}\begin{split}
\int_{M} \Delta_g R_g dv_g=&\lim_{\rho\rightarrow\infty} \int_{B^g(p, 2\rho)} \Delta_g R_g \eta_\rho dv_g \\
=&\lim_{\rho\rightarrow\infty} \int_{B^g(p, 2\rho)\setminus B^g(p, \rho)}  R_g \Delta_g\eta_\rho dv_g.\\
\end{split}\end{equation}

Let us assume for the moment that $M$ only has one end, and that the end is conformally flat, given by $E_1=(\mathbb R^4\setminus K, e^{2w}|dx|^2).$
By the Gauss equation, $|R_g| \leq 2 |L|^2$. Thus the above quantity is bounded by
\begin{equation}\begin{split}&\lim_{\rho\rightarrow\infty}( \int_{B^g(p, 2\rho)\setminus B^g(p, \rho)}  |L|^4 dv_g)^{1/2} \cdot
( \int_{B^g(p, 2\rho)\setminus B^g(p,\rho)}|\Delta_g\eta_\rho|^2 dv_g )^{1/2}\\
\leq& \lim_{\rho\rightarrow\infty}( \int_{B^g(p, 2\rho)\setminus B^g(p, \rho)}  |L|^4 dv_g)^{1/2} \cdot
(\frac{Vol_g(B^g(p, 2\rho)\setminus B^g(p, \rho))}{\rho^4} )^{1/2}.\\
\end{split}
\end{equation}
Here the volume is with respect to the metric $g$.
By the previous result in \cite{YW1}, if the metric is normal, and the $Q$-curvature is absolutely integrable, then there is  a quasiconformal mapping at infinity on each end, and thus $Vol_g(B^ g(p,r))\leq Cr^4$.

Therefore we have the volume growth estimate
$$
\frac{Vol_g(B^g(p, 2\rho)\setminus B^g(p, \rho))}{\rho^4} \leq C .$$
Then by the $L^4$-integrability assumption of the second fundamental form, the limit tends to $0$.

If $M$ has multiple ends, then
\begin{equation}\begin{split}&\int_{M^4}\Delta_gR_gd v_g\\
=&\lim_{\rho\rightarrow \infty}\int_{B^g(p,2\rho)\setminus B^g(p,\rho)} R_g \Delta_g \eta_\rho dv_g\\
=&\lim_{\rho\rightarrow \infty}\sum_{j=1}^N \int_{(B^g(p,2\rho)\setminus B^g(p,\rho))\cap E_j}R_g \Delta_g \eta_\rho dv_g.\\
\end{split}
\end{equation}
On each $E_j$, we can apply the above argument to obtain that
\begin{equation}\begin{split}&\lim_{\rho\rightarrow \infty}|\int_{(B^g(p,2\rho)\setminus B^g(p,\rho))\cap E_j} R_g \Delta_g \eta_\rho dv_g|\\
\leq &\lim_{\rho\rightarrow \infty}(\int_{(B^g(p,2\rho)\setminus B^g(p,\rho))\cap E_j} |L |^4dv_g)^{1/2}\\
&\hspace{55mm}\cdot(\int_{(B^g(p,2\rho)\setminus B^g(p,\rho))\cap E_j} |\Delta_g \eta_\rho|^2dv_g)^{1/2}\\
\leq& \lim_{\rho\rightarrow\infty}( \int_{(B^g(p,2\rho)\setminus B^g(p,\rho))\cap E_j}  |L|^4 dv_g)^{1/2}\\
& \hspace{50mm} \cdot (\frac{Vol_g((B^g(p, 2\rho)\setminus B^g(p, \rho))\cap E_j)}{\rho^4} )^{1/2}.\\
\end{split}
\end{equation}
As a direct corollary of Theorem 1.5 in \cite{YW1}, on each end $E_j$, we have
$$
\frac{Vol_g((B^g(p, 2\rho)\setminus B^g(p, \rho))\cap E_j)}{\rho^4} \leq C .$$
Therefore,
by the $L^4$-integrability assumption of the second fundamental form, the limit tends to $0$.

\end{proof}
\hide {
So $\int_{M^4}|Q_g| d v_g<\infty$. This implies $\int_{M^4}Q_g d v_g\leq 4\pi^2\chi(M)$ and $w=\int_{\mathbb R^n\setminus B}\log\frac{|y|}{|x-y|} * Q_ge^{4w} d y$. Therefore there is  a quasiconformal mapping at infinity, and thus $Vol_g(B_r)\leq Cr^n$.}

As we mentioned earlier, on a locally conformally flat manifold with finite total $Q$-curvature, there is no Ricci curvature lower bound. Therefore, we do not have the existence of cut-off functions with Hessian bound as described in the previous lemma. In order to overcome this difficulty, we will make use of the conformal structure
to obtain a different exhaustion of the manifold.

\begin{lemma}\label{vanishing}
Let $M$ be the $4$-dimensional manifold defined 
in Theorem~\ref{main theorem}.
Suppose the $Q$-curvature is absolutely integrable, the second fundamental form $L$ is in $L^4(M)$, and the metric is normal on each end.
Then $$\int_{M} \Delta_g R_g dv_g=0.$$
\end{lemma}

\begin{proof}[Proof of Lemma~\ref{vanishing}]
Let $B^0(0, \rho)$ be the ball centered at the origin, with radius $\rho$ with respect
to the Euclidean metric. On the Euclidean space, there always
exists a smooth cut-off function $\eta_\rho$ which is supported on $B^0(0, 2\rho)$.
It is equal to $1$ on $B^0(0, \rho)$, and its $k$-th derivative
is of order $O(1/\rho^k)$ over the annulus $B^0(0, 2\rho)\setminus B^0(0, \rho)$.
Again since the $Q$-curvature is absolutely integrable, so is $\Delta_gR_g$.

Suppose $M^4$ has one end $E_1$ first.
Let $\eta_\rho=1$ on $N$.
Then
\begin{equation}\label{2.3}
\begin{split}
&\int_{M}\Delta_gR_g d v_g\\
=&\lim_{\rho\rightarrow\infty} \int_{N\cup ( B^0(0, 2\rho) \cap E_1)}\Delta_gR_g \eta_\rho d v_g\\
=&\lim_{\rho\rightarrow\infty} \int_{B^0(0, 2\rho)\setminus B^0(0, \rho)}R_g \Delta_g \eta_\rho d v_g.\\
\end{split}\end{equation}
Here the last equality is because all boundary terms in the integration by parts formula vanish,
and $\Delta_g\eta_\rho=0$ on the complement of $B^0(0, 2\rho)\setminus B^0(0, \rho)$.

Using
$$ dv_g=e^{4w}dx,$$
$$\Delta_g\eta_\rho dv_g=\partial_i(e^{2w}\partial_i\eta_\rho) dx, $$ we have
\begin{equation} \begin{split}
\int_{B^0(0, 2\rho)\setminus B^0(0, \rho)}  R_g \Delta_g\eta_\rho dv_g=&
\int_{B^0(0, 2\rho)\setminus B^0(0, \rho)}  R_g \partial_i(e^{2w}\partial_i\eta_\rho )dx \\
=& \int_{B^0(0, 2\rho)\setminus B^0(0, \rho)}  R_g (\Delta_0\eta_\rho e^{2w}+ \partial_i(e^{2w})\partial_i\eta_\rho )dx \\
\leq&  C\int_{B^0(0, 2\rho)\setminus B^0(0, \rho)}   \frac{R_g}{\rho^2} e^{2w}dx\\
&   + C \int_{B^0(0, 2\rho)\setminus B^0(0, \rho)}  \frac{R_g |\partial_i w|}{\rho}e^{2w}dx \\
=:& I+II.\\
\end{split} \end{equation}
The first term $I$ can be bounded by the $L^4$-norm of the second fundamental form.
\begin{equation}\begin{split}
| I|\leq&C (\int_{B^0(0, 2\rho)\setminus B^0(0, \rho)} | R_g|^2 e^{4w}dx)^{1/2}\cdot
(\int_{B^0(0, 2\rho)\setminus B^0(0, \rho)} \frac{1}{\rho^4}dx)^{1/2}\\
\leq&C(\int_{B^0(0, 2\rho)\setminus B^0(0, \rho)} | L|^4 dv_g)^{1/2}\rightarrow 0,\\
\end{split}\end{equation}
as $\rho$ tends to $\infty$.

We will now study $II$ through the asymptotic behavior of the derivatives of $w$. We notice that the pointwise estimate of
$\partial_i w$ is not valid. But since we are taking the integral over the annulus (with respect to the Euclidean metric),
it can be reduced to the integral estimate of $\partial_i w$ over spheres at the end of the manifold.
\begin{equation}\begin{split}
|II|=&C \left|\int_{B^0(0, 2\rho)\setminus B^0(0, \rho)}  \frac{R_g \partial_i w}{\rho}e^{2w}dx\right|\\
\leq&C (\int_{B^0(0, 2\rho)\setminus B^0(0, \rho)} | R_g|^2 e^{4w}dx)^{1/2}\cdot
(\int_{B^0(0, 2\rho)\setminus B^0(0, \rho)} \frac{|\partial_i w|^2}{\rho^2}dx)^{1/2}\\
\leq&C(\int_{B^0(0, 2\rho)\setminus B^0(0, \rho)} | L|^4 dv_g)^{1/2}\cdot
(\int_{B^0(0, 2\rho)\setminus B^0(0, \rho)} \frac{|\partial_i w|^2}{\rho^2}dx)^{1/2}.\\
\end{split}\end{equation}
Notice that
\begin{equation}\begin{split}\label{3.1}
&  \int_{B^0(0, 2\rho)\setminus B^0(0, \rho)} |\partial_i w|^2dx\\
=&\int_{B^0(0, 2\rho)\setminus B^0(0, \rho)}\left|\frac1{4\pi^2}\int_{\mathbb R^4}\frac{x_i-y_i}{|x-y|^2}Q e^{4w(y)} d y\right|^2 d v_0\\
\leq&C \int_{B^0(0, 2\rho)\setminus B^0(0, \rho)} \left| \int_{\mathbb{R}^4}\frac{1}{|x-y|}Q(y)e^{4w(y)} dy\right|^2dx\\
\leq&C \int_{B^0(0, 2\rho)\setminus B^0(0, \rho)}\int_{\mathbb{R}^4}\frac{1}{|x-y|^2}  Q(y)e^{4w(y)} dydx \cdot \int_{\mathbb{R}^4} Q(y)e^{4w(y)} dy.\\
\end{split}\end{equation}
Since for any $y\in \mathbb{R}^4$, we have
$$\int_{x\in \partial B^0(0, r)} \frac{1}{|x-y|^2} d\sigma(x)=| \partial B^0(0, r)|\cdot O(\frac{1}{r^2}),$$
\begin{equation} \begin{split}\int_{B^0(0, 2\rho)\setminus B^0(0, \rho)}\frac{1}{|x-y|^2} dx=&\int_\rho^{2\rho}
\int_{x\in \partial B^0(0, r)} \frac{1}{|x-y|^2} d\sigma(x)dr\\
=&\int_\rho^{2\rho}| \partial B^0(0, r)|\cdot O(\frac{1}{r^2})
dr=O(\rho^2).\\\end{split}
\end{equation}
Plugging this into (\ref{3.1}), and using the fact that  $\int_{\mathbb{R}^4}Q(y)e^{4w(y)} dy<\infty$,
we obtain
\begin{equation}\begin{split}
\int_{B^0(0, 2\rho)\setminus B^0(0, \rho)} |\partial_i w|^2dx
\leq C (\int_{\mathbb{R}^4}Q(y)e^{4w(y)} dy)^2 \cdot O(\rho^2)= O(\rho^2).
\end{split}\end{equation}
Therefore,
\begin{equation}\begin{split}
| II|\leq&C(\int_{B^0(0, 2\rho)\setminus B^0(0, \rho)} | L|^4 dv_g)^{1/2}\cdot
(\frac{1}{\rho^2}\int_{B^0(0, 2\rho)\setminus B^0(0, \rho)} |\partial_i w|^2dx)^{1/2}\\
\leq&C(\int_{B^0(0, 2\rho)\setminus B^0(0, \rho)} | L|^4 dv_g)^{1/2}\rightarrow 0\\
\end{split}\end{equation}
as $\rho$ tends to $\infty$.
To conclude,
\begin{equation}\begin{split}
|\int_{M^4} \Delta_g R_g dv_g|=&\lim_{\rho\rightarrow\infty}| \int_{B^0(0, 2\rho)\setminus B^0(0, \rho)} R_g \Delta_g \eta_\rho dv_g| \\
\leq& \lim_{\rho\rightarrow\infty} |I|+|II|=0.\\  \end{split}\end{equation}
In general, $M$ has finitely many simple ends
$$M= N\bigcup \{\bigcup_j E_j \}.$$ We define $\eta_\rho$ to be
equal to 1 on $N$, $\eta_\rho =1$ on $E_j\cap B^0(0,\rho)$, $\eta_\rho =0$ on $E_j\setminus B^0(0,2 \rho)$, and its $k$-th derivative is of order $O (1/ \rho^ k)$ on the annulus $B ^ 0 (0, 2 \rho)\setminus B ^ 0 (0, \rho)$. Then \eqref{2.3} becomes
\begin{equation}\label{2.4}
\begin{split}
\int_{M^4}\Delta_gR_g d v_g=&\int_{N^4}\Delta_gR_g \eta_\rho d v_g+\sum_{j} \lim_{\rho\rightarrow\infty}\int_{B^0(0, 2\rho) \cap E_j}\Delta_gR_g \eta_\rho d v_g\\
=&\sum_{j} \lim_{\rho\rightarrow\infty} \int_{(B^0(0, 2\rho)\setminus B^0(0, \rho))  \cap E_j}  R_g \Delta_g\eta_\rho dv_g.
\end{split}\end{equation}
We now use the argument for manifold with only one end to show that on each end
\begin{equation*}
\lim_{ \rho\rightarrow\infty} \int_{(B^0(0, 2 \rho)\setminus B^0(0,  \rho))\cap E_j}  R_g \Delta_g \eta_\rho dv_g= 0
.\end{equation*}
This completes the proof of the lemma.
\end{proof}

\section{Proof of Theorem \ref{main theorem}}
\hide{1. The integral of the $Q$-curvature is an integral multiple of $4\pi^2$, when the ambient space is $R^5$. \\
2. The normal map extends continuously
to the end of the manifold.\\
}

We begin this section with a lemma asserting that there exists a sequence of domains such that the integral of the second fundamental form over the boundary tends to zero. This lemma is analogous to the lemma in~\cite{White}*{\S 2}. But unlike the 2-dimensional case, we do not have estimate of the area of the geodesic spheres. We circumvent this difficulty by exploring the conformal structure at the end.
\begin{lemma}\label{sequence of balls} 
Assume that  $M^n$ is an $n$-dimensional complete Riemannian manifold immersed in $\mathbb{R}^{n+1}$ with finitely many conformally flat simple ends and $\displaystyle \int_{M} |L|^n dv_g <+\infty$. Then on each end $E_j$ there exists a sequence $r_i\rightarrow \infty$, such that
$$\int_{\partial B^0(0,r_i)}|L|^{n-1} d\sigma_g \rightarrow0,$$
where $B^0(0,r_i)$ denotes the ball of radius $r_i$ with respect to the Euclidean distance; $d\sigma_g$ denotes
the area form on $\partial B^0(0,r_i)$ using the metric $g$.
\end{lemma}
\begin{remark}
Note that we do not need to assume the metric is normal in Lemma \ref{sequence of balls}.
\end{remark}

\begin{proof}[Proof of Lemma \ref{sequence of balls}]
On the end $E_j=\mathbb{R}^n \setminus B$,
\begin{equation}\begin{split}
\int_{\partial B^0(0,r)} |L|^{n-1} e^{(n-1)w}d\sigma_0 \leq &
\dstyle  \left( \int_{\partial B^0(0,r)} |L|^n e^{nw} d\sigma_0 \right )^{\frac{n-1}{n}} \cdot  \left(\int_{\partial B^0(0,r)} d\sigma_0\right )^{1/n}\\
=& C \dstyle  \left( \int_{\partial B^0(0,r)} |L|^n e^{nw} d\sigma_0 \right)^{\frac{n-1}{n}} \cdot r^{\frac{n-1}{n}},\\
\end{split}
\end{equation}
where $d\sigma_0$ denotes the area form of $\partial B^0(0,r)$ with respect to the Euclidean metric. $d\sigma_g=e^{(n-1)w}d\sigma_0$.

Thus
\begin{equation}
r^{-1} \dstyle \left( \int_{\partial B^0(0,r)} |L|^{n-1} e^{(n-1)w} d\sigma_0 \right)^{\frac{n}{n-1}} \leq C  \int_{\partial B^0(0,r)} |L|^n e^{nw} d\sigma_0 .
\end{equation}
On the $j$-th end, let $r_0$ being the smallest number such that $B \subset B ^ 0 (0, r _0)$.
We now integrate $r$ between $[r_0, +\infty)$,
\begin{equation}\label{seqn}
\begin{split}
\dstyle \int_{r_0}^{\infty} r^{-1} \dstyle \left( \int_{\partial B^0(0,r)} |L|^{n-1} e^{(n-1)w} d\sigma_0 \right)^{\frac{n}{n-1}} dr
\leq & C \int_{\mathbb{R}^n\setminus B} |L|^n e^{nw} dx\\
\leq& C   \int_{M} |L|^n dv_g<\infty.\\
\end{split}
\end{equation}
Therefore, there exists a sequence $\{r_i\}\rightarrow \infty$ such that $$\int_{\partial B^0(0,r_i)}|L|^{n-1} d\sigma_g \rightarrow0,$$
because if not, the left hand side of (\ref{seqn}) is not integrable.
\end{proof}

\begin{lemma}\label{det}
Let $M$ be the manifold defined in the previous lemma. 
Let $\vec n$ be the Gauss map
$
M\to \mathbb S^n.
$
Assume that   $\displaystyle \int_M |L|^n dv_g<+\infty$. Then
\begin{equation}
\displaystyle \int_{M^n} \det (d \vec n)= |\mathbb{S}^n|\cdot m
\end{equation}
for some integer $m$, where $|\mathbb S^n|$ is the volume of the unit sphere.
\end{lemma}
\begin{proof} Fix an integer $j $. By Lemma \ref{sequence of balls}, on the $j$-th end $E_j=\mathbb{R}^n\setminus B$, there exists a sequence of Euclidean balls
$B^0(0, r_i^j)$, $r_i^j\rightarrow \infty$ as $i\rightarrow \infty$, such that

$$\int_{\partial B^0(0,r_i^j)}|L|^{n-1} d\sigma_g \rightarrow0.$$
For this fixed $j$, the image of $\partial B^0(0,r_i^j)$ under the Gauss map is a set of closed $(n-1)-$dimensional submanifolds. Notice that
the second fundamental form $L$ can be regarded as the differntial of the Gauss map. Thus by change of
variable under the Gauss map,
\begin{equation}
Area_{\mathbb{S}^{n}}\left(\vec n(\partial B^0(0,r_i^j))\right)\leq (\epsilon_i^j)^{(n-1)}
\end{equation}
where for each fixed $j$, $\eps_i^j\to 0$ as $i\to\infty$ is a sequence of positive numbers.
Here the area is measured by the standard metric of $\mathbb{S}^{n}$.
By the isoperimetric inequality of $\mathbb{S}^{n}$, there exist disks $D_i^j$, enclosed by the image of $\partial B^0(0,r_i^j)$ under the Gauss
map, such that
\begin{equation}
Vol_{\mathbb{S}^{n}}(D_i^j) \leq O((\epsilon_i^ j)^n).
\end{equation}

We recall that
\begin{equation*}
M =N^n \bigcup \left\{ \bigcup_{j=1}^N E_j \right\},
\end{equation*}
where $(N^n, g)$ is a compact manifold with boundary,
and each $E_j$ is a conformally flat simple end of $M^n$; that is
\begin{equation*}
(E_j, g)=  (\mathbb R^n \setminus B, e^{2w} |dx |^ 2).
\end{equation*}
Without loss of generality, we assume that $r_i^j$ is big enough such that $E_j\bigcap B^0(0,r_i^ j)= B^0(0,r_i^ j).$ Then
\begin{equation}\begin{split}
&\dstyle \int_{M}	 \det (d \vec n) dv_g\\
=&\lim_{i\rightarrow\infty} \dstyle \int_{N^n \bigcup \{ \bigcup_j E_j\bigcap B^0(0,r_i^ j)\}} \det (d \vec n) dv_g\\
=&\lim_{i\rightarrow\infty} \int_{N^n \bigcup \{ \bigcup_j B^0(0,r_i^ j)\}}  \vec n^{\sharp}(\omega) =\lim_{i\rightarrow\infty} \int_{\vec n( N^n \bigcup \{ \bigcup_j B^0(0,r_i^ j)\})} \omega\\
=& \lim_{i\rightarrow\infty} \int_{\vec n( N^n \bigcup \{ \bigcup_j B^0(0,r_i^ j)\})\setminus \bigcup_j D_i^j}\omega+\int_{\bigcup_j D_i^j}\omega.\\
\end{split}\end{equation}
Here we denote the standard volume form on $\mathbb{S}^{n}$ by $\omega$. \\
Since $$\vec n( N^n \bigcup \{ \bigcup_j B^0(0,r_i^ j)\}) \quad \mbox{and} \quad \bigcup_i D_i^ j$$ have the same boundary, $\vec n( N^n \bigcup \{ \bigcup_j B^0(0,r_i^ j)\})\setminus \bigcup_j D_i^j$ is an integral cycle. Thus for
some integer $m$,
\begin{equation}
\dstyle \int_{\vec n( N^n \bigcup \{ \bigcup_j B^0(0,r_i^ j)\})\setminus \bigcup_j D_i^j} \omega=m \int_{\mathbb{S}^n} \omega =|\mathbb S^n| m.
\end{equation}
Thus
\begin{equation}
\dstyle \int_{M^n}  \det (d \vec n) dv_g -|\mathbb S^n| m= \int_{\bigcup_j D_i^ j} \omega =O((\epsilon_i^j)^n)\rightarrow 0.
\end{equation}
This finishes the proof that
\begin{equation}
\dstyle \int_{M^n}  \det (d \vec n) dv_g= |\mathbb S^n| m.
\end{equation}

\end{proof}
\begin{proof}[Proof of Theorem \ref{main theorem}]
By computing $A_g=\frac{1}{2}(Ric-\frac{R}{6}g) $
using the Gauss equation, we observe that
$$\det (d \vec n) = \frac23\sigma_2 (A_g) $$
on four dimensional immersed manifold $M\hookrightarrow \mathbb{R}^5$.

Now we recall the definition of $Q$-curvature on four manifold
\begin{equation}\begin{split}
Q=&\frac{1}{12}(-\Delta_g R_g+\frac{1}{4} R^2 -3|E|^2)\\
=&-\frac{1}{12}\Delta_g R_g+ 2\sigma_2 (A_g),
\end{split}
\end{equation}
where $E$ is the traceless part of the Ricci curvature. 
Thus $Q$-curvature differs from $\sigma_2$ of Schouten tensor by a divergence term. By Lemma \ref{vanishing}, the integral of
the Laplacian of the scalar curvature $R_g$ vanishes if $Q$-curvature is totally integrable. Therefore,
\begin{equation}\begin{split}
\int_M Q_g dv_g=&
\int_M  2\sigma_2 (A_g) dv_g=3\int_M\det(d\vec n)d v_g.\\\end{split}
\end{equation}
By Lemma \ref{det}, this is equal to
$$3|\mathbb S^4| m=3\cdot \frac{8\pi^2}3 m=8\pi^2 m.$$

\end{proof}

\begin{remark}
In the above proof, we have proved that $\displaystyle \int_M \sigma_2 (A_g) dv_g$ is an integral multiple of $4\pi^2$.
\end{remark}

\hide{\bf Theorem 1.4
Let $(M^n,g)$ be an even dimensional locally conformally flat complete
manifold with finite total $Q$-curvature and finitely many conformally flat simple ends.
Suppose that on each end, the metric is normal. If $M^n$ is immersed in $\mathbb{R}^{n +1}$ with
\begin{equation}
\dstyle \int_{M^n} |L|^ndv_g< \infty,
\end{equation}
with $L$ being the second fundamental form,
then
$$\int_{M^n} Q_g dv_g=c_n\mathbb{Z},$$
where $c_n= 2 ^ {n -2} (\frac n-2 2)$. It is equal to the integral of the $Q$-curvature on the standard $n $-sphere $\mathbb{S}^n$.}

\section {Proof of Theorem \ref{mainhigh}}

The theorem is also valid for all even dimensional locally conformally flat manifolds with
simple ends,
if the metric on each end is normal. We begin by the following lemma, which seems to be well-known.

\begin{lemma}\label{Pfaffian high}
$$\mathrm{Pfaff(\Omega)}= (n -1)!!\det(d \vec n). $$
\end{lemma}
\begin{proof}
For higher dimension, the relation between $\det(d \vec n)$ and the Pfaffian of a Riemannian curvature $\mathrm{Pffaf(\Omega)}$ is given by the following formula:
suppose $\{e_i, 1\leq i\leq n\}$ is a locally orthogonal frame whose coframe field is $\{\theta_i, 1\leq i\leq n\}$, the curvature form $\Omega_{ij}=\frac12 R_{ijkl} \theta^k\wedge \theta^l$.
Consider the differential $n$-form
\begin{equation}\Omega=(-1)^{\frac n2}\frac{1}{2^n\pi^{\frac n2}(\frac n2)!}\delta_{1\cdots n}^{i_1\cdots i_n}\Omega_{i_1i_2}\wedge\cdots\wedge\Omega_{i_{n-1}i_n}.\end{equation}
$\Omega$ can be denoted by
\begin{equation}\Omega=K d\sigma,\end{equation} where $d\sigma=\theta^1\wedge\cdots\wedge \theta^n$.
Here
\begin{equation}K=\frac 1{2^n(2\pi)^{\frac n2}(\frac n2)!}\delta_{j_1\cdots j_n}^{i_1\cdots i_n}R_{i_1i_2j_1j_2}\cdots R_{i_{n-1}i_nj_{n-1}j_n}.\end{equation}
We note that
\begin{equation}\label{Pfaffian}
\begin{aligned}
\mathrm{Pfaff}(\Omega)=&(2\pi)^{\frac n2}K\\
=&(2\pi)^{\frac n2}\frac 1{2^n(2\pi)^{\frac n2}(\frac n2)!}\delta_{j_1\cdots
j_n}^{i_1\cdots i_n}R_{i_1i_2j_1j_2}\cdots R_{i_{n-1}i_nj_{n-1}j_n}.
\end{aligned}\end{equation}
By the Gauss equation $R_{ijkl}=L_{ik}L_{jl}-L_{il}L_{jk}$, and the fact that $\delta_{j_1j_2\cdots j_n}^{i_1i_2\cdots i_n}=-\delta_{j_2j_1\cdots j_n}^{i_1i_2\cdots i_n}$, we obtain
\begin{equation}
\begin{aligned}
\mathrm{Pfaff}(\Omega)=&\frac1{2^n(\frac n2)!}\delta_{j_1\cdots j_n}^{i_1\cdots i_n}(L_{i_1j_1}L_{i_2j_2}-L_{i_1j_2}L_{i_2j_1})\\
&\hspace{30mm}\cdots(L_{i_{n-1}j_{n-1}}L_{i_nj_n}-L_{i_{n-1}j_n}L_{i_nj_{n-1}})\\
=&\frac1{2^n(\frac n2)!}\delta_{j_1\cdots j_n}^{i_1\cdots i_n}(2L_{i_1j_1}L_{i_2j_2})\cdots(2L_{i_{n-1}j_{n-1}}L_{i_nj_n})\\
=&\frac{n!}{2^{\frac n2}n!(\frac n2)!}\delta_{j_1\cdots j_n}^{i_1\cdots i_n}L_{i_1j_1}L_{i_2j_2}\cdots L_{i_{n-1}j_{n-1}}L_{i_nj_n}\\
=&\frac{n!}{2^{\frac n2}(\frac n2)!}\det(L)\\
=&(n-1)!!\det(L).
\end{aligned}
\end{equation}
\end {proof}

The integration of $\mathrm{Pfaff}(\Omega)$ or $K$ appears in the Gauss-Bonnet-Chern theorem:
\begin{equation}\frac1{(2\pi)^{\frac n2}}\int_M\mathrm{Pfaff}(\Omega)=\int_MK=\chi(M).\end{equation}
When $n=4$,
\begin{equation}\mathrm{Pfaff}(\Omega)=2\sigma_2(A_g)=Q_g+\frac 1{12}\Delta_gR_g.\end{equation}

Next we prove an analogous result of Lemma \ref{vanishing}. However, due to the complexity of higher dimensions, the proof is more complicated, and  the conformally flat structure is used in an essential way.
\begin{lemma}\label{vanishing2} Suppose $M^n$ satisfies the assumptions in Theorem \ref{mainhigh}. $T^i$ is an intrinsic vector field of weight\footnote{See also~\cite{lu}*{Page 245}.} $(-n+1)$ on $M^n$. Then
\begin{equation*}\int_ M div_i T^i(g)dv_g = 0. \end{equation*}
\end{lemma}
\begin{proof}By a classical result which is essentially due to
Weyl \cite{Weyl},
an intrinsic vector field $T^i(g)$ is a linear combination
$$T^i(g)= \sum_{q\in Q} a_q C^{q,i}(g),$$
which each $C^{q,i}(g)$ is a partial contraction with one free index $i$ that takes the form
$$C^{q,i}(g)=pcontr_i(\nabla^{(m_1)}_{r_1\dots r_{m_1}}R_{i_1j_1k_1l_1} \otimes \cdots \otimes \nabla^{(m_a)}_{t_1\dots t_{m_a}}R_{i_aj_ak_al_a})$$
with $\displaystyle \sum_{t=1}^a (m_t+2)=n-1$.
We simplify notations to write it in the following form:
$$C^{q,i}(g)=pcontr_i(\nabla^{(m_1)}Rm \otimes \cdots \otimes \nabla^{(m_a)}Rm).$$
The main idea of the following proof is to factor out a curvature term $\displaystyle \int |Rm|^{n/2}dv_g$, which can be controlled by $C \displaystyle \int |L|^n dv_g $ (see \eqref{eqn:Ln}), and then to estimate the curvature derivative terms in terms of the conformal factor $w$. We begin with the following formula of integration by parts. 

\begin{equation}\label{eqn:4.0}
\begin{split}
&\int_{M} div_i C^{q,i}(g) dv_g\\
=&-\lim_{\rho\rightarrow \infty}  \sum_{j=1}^{N} \int_{E_j \cap (B^0(2\rho)\setminus B^0(\rho)) } C^{q,i}(g) (\eta_\rho)_i dv_g\\
=&-\lim_{\rho\rightarrow \infty}  \sum_{j=1}^{N} \int_{E_j \cap (B^0(2\rho)\setminus B^0(\rho)) }
pcontr_i(\nabla^{(m_1)}Rm \otimes \\
&\hspace{50mm}\cdots \otimes \nabla^{(m_a)}Rm)(\eta_\rho)_i dv_g,\\
\end{split}
\end{equation}
where $(\eta_\rho)_i$, which is defined in the proof of Lemma~\ref{vanishing}, is  supported on $B^0(2\rho)\setminus B^0(\rho)$.
On each end, we consider the term
\begin{equation}
\label{eqn:4.9}\int_{E_j \cap (B^0(2\rho)\setminus B^0(\rho)) }
pcontr_i(\nabla^{(m_1)}Rm \otimes \cdots \otimes \nabla^{(m_a)}Rm)(\eta_\rho)_i dv_g.\end{equation}
Using integration by parts, we obtain a term $R_{ijkl}$ without any derivative in the contraction. If there is already such a term in the contraction, then we skip this step.
Then \eqref{eqn:4.9} is equal to
\begin{equation}\begin{split}
\label{eqn:4.10}
&\int_{E_j \cap (B^0(2\rho)\setminus B^0(\rho)) }
(-1)^{m_1}R_{i_1j_1k_1l_1}\otimes \\
&  pcontr_{i_1j_1k_1l_1}( \nabla^{(m_1)}[(\nabla^{(m_2)}Rm\otimes \cdots \otimes \nabla^{(m_a)}Rm)(\eta_\rho)_i]) dv_g.\\
\end{split}\end{equation}
We use $pcontr_{i_1j_1k_1l_1}$ to denote the partial contraction with four free indices $i_1, j_1, k_1, l_1$.

By the conformal change $g=e^{2w}|dx|^2$ and the H\"{o}lder's inequality, \eqref{eqn:4.10} is bounded
by
\begin{equation}\begin{split}
\label{eqn:4.11}
&|\int_{E_j \cap (B^0(2\rho)\setminus B^0(\rho)) }
(-1)^{m_1}R_{i_1j_1k_1l_1}e^{2w}\\
&  pcontr_{i_1j_1k_1l_1}\left( e^{-2w}\nabla^{(m_1)}[(\nabla^{(m_2)}Rm\otimes \cdots \otimes \nabla^{(m_a)}Rm)(\eta_\rho)_i]\right) e^{nw}dx|\\
\leq &\left(\int_{E_j \cap (B^0(2\rho)\setminus B^0(\rho)) }
|Rm|^{n/2} e^{nw}dx\right)^{2/n} \\
& \cdot \left(\int_{E_j \cap (B^0(2\rho)\setminus B^0(\rho)) }  \Big|e^{-2w} pcontr_{i_1j_1k_1l_1}\left(\nabla^{(m_1)}[(\nabla^{(m_2)}
Rm\otimes \right. \right. \\
&\hspace{30mm}\left.\left. \cdots \otimes \nabla^{(m_a)}Rm)(\eta_\rho)_i]\right)e^{nw}\Big|^{\frac{n}{n-2}} dx\right)^{\frac{n-2}{n}}.\\
\end{split}\end{equation}

\noindent Since
\begin{equation}\label{eqn:Ln}\begin{split}
&\int_{E_j \cap (B^0(2\rho)\setminus B^0(\rho)) }
|Rm|^{n/2} e^{nw}dx\\
=&\int_{E_j \cap (B^0(2\rho)\setminus B^0(\rho)) }
|Rm|^{n/2} dv_g\\
\leq& C\int_{E_j \cap (B^0(2\rho)\setminus B^0(\rho)) }
|L|^{n} dv_g\rightarrow 0,  \\
\end{split}\end{equation}
as $\rho\rightarrow \infty$, we reduce the problem to show that
\begin{equation}\begin{split}\label{eqn:4.12}
&\int_{E_j \cap (B^0(2\rho)\setminus B^0(\rho)) } 
 \Big| pcontr_{i_1j_1k_1l_1}\left(\nabla^{(m_1)}[(\nabla^{(m_2)}
Rm\otimes \right.\\
&\hspace{20mm}\left.\cdots \otimes \nabla^{(m_a)}Rm)(\eta_\rho)_i]\right)\Big|^{\frac{n}{n-2}}e^{nw} dx
\leq C.\\
\end{split}
\end{equation}

\noindent We use $g=e^{2w}|dx|^2$ to write the integrand
\begin{equation}\label{eqn:4.13}\Big|pcontr_{i_1j_1k_1l_1}\left( \nabla^{(m_1)}[(\nabla^{(m_2)}
Rm\otimes
\cdots \otimes \nabla^{(m_a)}Rm)(\eta_\rho)_i]\right)\Big|^{\frac{n}{n-2}}e^{nw}\end{equation}
in coordinate derivatives of $w$.

Note that under the flat  coordinate system, we have
\begin{equation}\begin{split}
\Gamma_{jk}^s=& \frac{1}{2} g^{sl}(\frac{\partial g_{lj}}{\partial x_k}+\frac{\partial g_{lk}}{\partial x_j}-\frac{\partial g_{jk}}{\partial x_l})  \\
=&\frac{\partial w}{\partial x_{k}}\delta_{sj}+\frac{\partial w}{\partial x_{j}}\delta_{sk}-\frac{\partial w}{\partial x_{s}}\delta_{jk},\\
\end{split}\end{equation}
and we have
\begin{equation}\begin{split}
&R_{ivku}g^{ul}g^{jv}=\Gamma_{ik}^p\Gamma_{vp}^lg^{jv} + \frac{\partial \Gamma_{ik}^l}{\partial x_v}       g^{jv}
-\Gamma_{vk}^p\Gamma_{ip}^{l}g^{jv}-\frac{\partial \Gamma_{vk}^l}{\partial x_i}g^{jv} \\
=&(\frac{\partial w}{\partial x_{k}}\delta_{ip}+\frac{\partial w}{\partial x_{i}}\delta_{kp}-\frac{\partial w}{\partial x_{p}}\delta_{ik})(\frac{\partial w}{\partial x_{j}}\delta_{pl}+\frac{\partial w}{\partial x_{p}}\delta_{lj}-\frac{\partial w}{\partial x_{l}}\delta_{jp})e^{-2w}\\
-&(\frac{\partial w}{\partial x_{k}}\delta_{jp}+\frac{\partial w}{\partial x_{j}}\delta_{kp}-\frac{\partial w}{\partial x_{p}}\delta_{jk})(\frac{\partial w}{\partial x_{i}}\delta_{pl}+\frac{\partial w}{\partial x_{p}}\delta_{li}-\frac{\partial w}{\partial x_{l}}\delta_{ip})e^{-2w}\\
+&(\frac{\partial^2 w}{\partial x_j \partial x_k}\delta_{il}+ \frac{\partial^2 w}{\partial x_i \partial x_j}\delta_{kl}-\frac{\partial^2 w}{\partial x_j \partial x_l}\delta_{ik}  )e^{-2w}\\
-&(\frac{\partial^2 w}{\partial x_i \partial x_k}\delta_{jl}+ \frac{\partial^2 w}{\partial x_i \partial x_j}\delta_{kl}-\frac{\partial^2 w}{\partial x_i \partial x_l}\delta_{jk}  )e^{-2w}\\
=&e^{-2w}\sum_{\alpha} b_{\alpha}\cdot \partial^{\alpha^1}w \cdots \partial^{\alpha^p}w,\\
\end{split}\end{equation}
where $$ \displaystyle \sum_{k=1}^{p}|\alpha^k|= 2.$$
Similarly,
\begin{equation}\begin{split}
&\nabla^{(m)}_{r_1\dots r_m} R_{ivku}g^{ul}g^{jv}\\
=&e^{-2w} \sum_{\alpha} b_{\alpha}\cdot \partial^{\alpha^1}w \cdots \partial^{\alpha^p}w.\\
\end{split}\end{equation}
where $$ \displaystyle \sum_{k=1}^{p}|\alpha^k|= m+2.$$

Notice that in the integrand of \eqref{eqn:4.12},
some of the derivatives in $\nabla^{(m_1)}$ fall on
$\nabla^{(m_2)}Rm\otimes \cdots \otimes \nabla^{(m_a)}Rm$, and others fall on $(\eta_\rho)_i$.
Denote the number of derivatives fall on $(\eta_\rho)_i$ by $n_1\geq0$.
By the definition of $\eta_\rho$, $\partial^{n_1+1}(\eta_\rho)\leq O(\frac{1}{\rho^{n_1+1}})$ on
$B^0(2\rho)\setminus B^0(\rho)$.

It is not hard to see that the integrand of \eqref{eqn:4.13} is bounded by
a finite linear combination of partial contractions of coordinate derivatives of $w$:
$$\sum_{n_1=0}^{m_1}\sum_{\alpha} b_{\alpha, n_1}\cdot |\partial^{\alpha^1}w \cdots \partial^{\alpha^p}w|^{\frac{n}{n-2}}
\cdot O( \frac{1}{\rho^{\frac{(n_1+1)n}{n-2}}}),$$
where the multi-index derivative is with respect to the Euclidean metric, defined by
$$\partial^{\alpha^k}w
=\frac{\partial^{\alpha^k_1}}{\partial x_1^{\alpha^k_1}}\cdots \frac{\partial^{\alpha^k_n}}{\partial x_n^{\alpha^k_n}}w.
$$

The indices $\alpha^1,...,\alpha^p$ satisfy
$$ \sum_{k=1}^{p}|\alpha^k|= n-n_1-3,$$
where the norm of a multi-index $\alpha^k$ is defined by $|\alpha^k| := \displaystyle \sum_{j=1}^n \alpha^k_j$.
We note that any powers of $e^{w}$ are cancelled in the process of the partial contraction.
And $$(n_1+1)\cdot\frac{n}{n-2}\geq\frac{n}{n-2}.$$


To prove the estimate \eqref{eqn:4.12}, 
we will show that for each $j=1,...,N$, $n_1=0,\cdots,m_1$, and multi-index $\alpha$, if
$$ \sum_{k=1}^{p}|\alpha^k|= n-2-a, \quad a=n_1+1\geq 1$$
then
\begin{equation}\label{eqn:4.1}\lim_{\rho\rightarrow \infty}  \int_{E_j \cap (B^0(2\rho)\setminus B^0(\rho))}
|\partial^{\alpha^1}w \cdots \partial^{\alpha^p}w|^{\frac{n}{n-2}} \cdot O(\frac{1}{\rho^{a\cdot \frac{n}{n-2}}}) dx\leq C.
\end{equation}

In order to prove \eqref{eqn:4.1}, we \underline{\bf{claim}} that
for $\alpha^1$,..., $\alpha^p$ satisfying $$\sum_{k=1}^p|\alpha^k| =n-2-a, \quad a\geq 1$$
there exists $q_1, ..., q_p$, such that
\begin{enumerate}
\item $q_1>1, ..., q_p>1$, for $k=1,...,p$;
\item $\displaystyle \frac{1}{q_1}+\cdots +\frac{1}{q_p}=1$;
\item $|\alpha^k|\cdot q_k<(n-1) \cdot \frac{n-2}{n} $, for $k=1,...,p$.
\end{enumerate}

Let us first apply the Claim to prove Lemma \ref{vanishing2}. Using the above Claim, we have
\begin{equation}\label{eqn:4.5}
\begin{split}
&\frac{1}{\rho^{a\cdot \frac{n}{n-2}}}\int_{E_j \cap (B^0(2\rho)\setminus B^0(\rho)) }
|\partial^{\alpha^1}w \cdots \partial^{\alpha^p}w|^{\frac{n}{n-2}}  dx\\
 &\leq \frac{1}{\rho^{a\cdot \frac{n}{n-2}}}
(\int_{E_j \cap (B^0(2\rho)\setminus B^0(\rho)) }
|\partial^{\alpha^1}w |^{q_1\cdot \frac{n}{n-2}}dx)^{\frac{1}{q_1}} \\
& \hspace{10mm}\cdots
(\int_{E_j \cap (B^0(2\rho)\setminus B^0(\rho)) }
|\partial^{\alpha^p}w |^{q_p\cdot \frac{n}{n-2}}dx)^{\frac{1}{q_p}},\\
\end{split}
\end{equation}
where $a=n_1+1\geq 1$.
By definition of normal metric,
\begin{equation}
w(x)=\frac{1}{c_n}\int_{\mathbb{R}^n\setminus B}\log\frac{|y|}{|x-y|}Q_g(y)e^{nw(y)} dy + C.
\end{equation}
Hence
\begin{equation}
|\partial^{\alpha^k} w(x)|\leq \frac{1}{c_n}\int_{\mathbb{R}^n\setminus B}\frac{1}{|x-y|^{|\alpha^k|}}|Q_g(y)|e^{nw(y)} dy,
\end{equation}
where $|\alpha^k| =\sum_{j=1}^n \alpha^k_j$.
By H\"{o}lder's inequality, for $q_k>1$,
\begin{equation}\begin{split}
|\partial^{\alpha^k} w(x)|^{q_k \cdot \frac{n}{n-2}}\leq& \frac{1}{c_n^{q_k \cdot \frac{n}{n-2}}}\int_{\mathbb{R}^n\setminus B}\frac{1}{|x-y|^{|\alpha^k|\cdot q_k \cdot \frac{n}{n-2}}}|Q_g(y)|e^{nw(y)} dy \\
&\cdot (\int_{\mathbb{R}^n\setminus B}|Q_g(y)|e^{nw(y)} dy)^{q_k \cdot \frac{n}{n-2}-1}.\\
\end{split}
\end{equation}
Since $	\int_{\mathbb{R}^n\setminus B} | Q_g(y)|e^{nw(y)} dy<\infty$, this is bounded by
$$C\int_{\mathbb{R}^n\setminus B}\frac{1}{|x-y|^{|\alpha^k|\cdot q_k \cdot \frac{n}{n-2}}}|Q_g(y)|e^{nw(y)} dy.$$
Therefore,
\begin{equation}\label{eqn:4.3}
\begin{split}
&\int_{\partial B^0(0,r)} |\partial^{\alpha^k} w(x)|^{q_k \cdot \frac{n}{n-2}} d\sigma_0(x) \\
\leq& C\int_{\partial B^0(0,r)} \int_{\mathbb{R}^n\setminus B}\frac{1}{|x-y|^{|\alpha^k|\cdot q_k \cdot \frac{n}{n-2}}}|Q_g(y)|e^{nw(y)} dyd\sigma_0(x).\\
\end{split}
\end{equation}
By condition (3) in the Claim, $n-1>|\alpha^k|\cdot q_k  \cdot \frac{n}{n-2}$. Then by using the homogeneity of the integral, we have
$$\int_{\partial B^0(0,r)}\frac{1}{|x-y|^{|\alpha^k|\cdot q_k \cdot \frac{n}{n-2}}}d\sigma_0(x)\leq  |\partial
B^0(0,r)| O(\frac{1}{r^{|\alpha^k|\cdot q_k \cdot \frac{n}{n-2}}}).
$$
Plugging the above inequality into \eqref{eqn:4.3},  we obtain
\begin{equation}\label{eqn:4.4}
\begin{split}
&\int_{\partial B^0(0,r)} |\partial^{\alpha^k} w(x)|^{q_k \cdot \frac{n}{n-2}} d\sigma_0(x) \\
\leq& Cr^{n-1-|\alpha^k|\cdot q_k \cdot \frac{n}{n-2}} \int_{\mathbb{R}^n\setminus B}    |Q_g(y)|e^{nw(y)} dy\\
=&\tilde{C} r^{n-1-|\alpha^k|\cdot q_k \cdot \frac{n}{n-2}}.\\
\end{split}
\end{equation}
Hence
\begin{equation}\label{eqn:4.6}
\begin{split}
&\int_{E_j\cap (B^0(0,2\rho)\setminus B^0(0,\rho))} |\partial^{\alpha^k} w(x)|^{q_k \cdot \frac{n}{n-2}} dx\\
\leq&\int_{\rho}^{2\rho}\int_{\partial B^0(0,r)} |\partial^{\alpha^k} w(x)|^{q_k \cdot \frac{n}{n-2}} d\sigma_0(x) dr \\
\leq& C\rho^{n-|\alpha^k|\cdot q_k \cdot \frac{n}{n-2}}.\\
\end{split}
\end{equation}
Using \eqref{eqn:4.6} in \eqref{eqn:4.5}, we obtain
\begin{equation}\label{eqn:4.7}
\begin{split}
&\frac{1}{\rho^{a \cdot \frac{n}{n-2}}}\int_{E_j \cap (B^0 (2\rho)\setminus B^0 (\rho))} |\partial^{\alpha^1} w\cdots \partial^{\alpha^p }w |^{\frac{n}{n-2}} dx\\
\leq& \frac{C}{\rho^{a \cdot \frac{n}{n-2}}} \rho^{\frac{n-|\alpha^1|\cdot q_1 \cdot \frac{n}{n-2}}{q_1}+\cdots+\frac{n-|\alpha^p|\cdot q_p \cdot \frac{n}{n-2}}{q_p}}.\\
\end{split}
\end{equation}
By the {\bf{claim}}, $\displaystyle \frac{1}{q_1}+\cdots +\frac{1}{q_p}=1$.
So \begin{equation}
\begin{split}
&\frac{n-|\alpha^1|\cdot q_1 \cdot \frac{n}{n-2}}{q_1}+\cdots+\frac{n-|\alpha^p|\cdot q_p \cdot \frac{n}{n-2}}{q_p}\\
=&n-  \frac{n}{n-2} \sum_{k=1}^p |\alpha^k|=n- \frac{n}{n-2}\cdot (n-2-a)=a  \cdot \frac{n}{n-2}.\\
\end{split}
\end{equation}
Thus \eqref{eqn:4.7} becomes
\begin{equation}\label{eqn:4.8}
\frac{1}{\rho^{a \cdot \frac{n}{n-2}}}\int_{E_j \cap (B^0 (2\rho)\setminus B^0 (\rho))} |\partial^{\alpha^1} w\cdots \partial^{\alpha^p }w |^{\frac{n}{n-2}}  dx
\leq C.\\
\end{equation}This completes the proof of Lemma \ref{vanishing2}.
\end{proof}

\begin{proof}[Proof  of \underline{\bf{claim}}]
Let $q_k= \frac{(n-1) \cdot \frac{n-2}{n}}{|\alpha^k|}-\epsilon$ for $k=1,...,p-1$. 
Let $q_p$ be defined by
\begin{equation}
\frac{1}{q_p}=1-\frac{1}{\frac{(n-1) \cdot \frac{n-2}{n}}{|\alpha^1|}-\epsilon}-\cdots-\frac{1}{\frac{(n-1) \cdot \frac{n-2}{n}}{|\alpha^{p-1}|}-\epsilon}.
\end{equation}
We will choose the value of $\epsilon$ later.
It is obvious that $q_1,...,q_p$ satisfy condition (1), (2) in the claim, 
and $q_1,...,q_{p-1}$ satisfy condition (3). We now show $q_p$ satisfies condition (3) if $\epsilon$ is chosen small enough.
To prove \begin{equation}
\begin{split}
&|\alpha^p|\cdot q_p\\
=& [(n-2-a)-\sum_{k=1}^{p-1}|\alpha^{k}|]\cdot q_p<(n-1) \cdot \frac{n-2}{n},
\end{split}
\end{equation}
It is the same to prove 
\begin{equation}
\begin{split}&(n-2-a)-\sum_{k=1}^{p-1}|\alpha^{k}|\\
<& (n-1)\cdot \frac{n-2}{n}(1-\frac{1}{\frac{(n-1)\cdot \frac{n-2}{n}}{|\alpha^1|}-\epsilon}-\cdots-\frac{1}{\frac{(n-1)\cdot \frac{n-2}{n}}{|\alpha^{p-1}|}-\epsilon}),\\
\end{split}\end{equation}
i.e. 
\begin{equation}\label{eqn:4.32}
\begin{split}
&n-2-a-\sum_{k=1}^{p-1}|\alpha^{k}|\\
<& (n-1)\cdot \frac{n-2}{n}-(\frac{1}{\frac{1}{|\alpha^1|}-\frac{\epsilon}{(n-1)\cdot \frac{n-2}{n}}}+\cdots+\frac{1}{\frac{1}{|\alpha^{p-1}|}-\frac{\epsilon}{(n-1)\cdot \frac{n-2}{n}}}).\\
\end{split}\end{equation}
For each $k=1,...,p-1$,
\begin{equation}\frac{1}{\frac{1}{|\alpha^k|}-\frac{\epsilon}{n-1}}\leq |\alpha^k|+O(\epsilon).\end{equation}
Thus 
\begin{equation}
\frac{1}{\frac{1}{|\alpha^1|}-\frac{\epsilon}{(n-1)\cdot \frac{n-2}{n}}}+\cdots+\frac{1}{\frac{1}{|\alpha^{p-1}|}-\frac{\epsilon}{(n-1)\cdot \frac{n-2}{n}}}\leq \sum_{k=1}^{p-1}|\alpha^{k}|+O(\epsilon).
\end{equation}
Since $n-2-a\leq n-3<(n-1) \cdot \frac{n-2}{n}$,
we can choose $\epsilon$ small enough, such that,
\begin{equation}n-2-a < (n-1)\cdot \frac{n-2}{n}-O(\epsilon).\end{equation}
This proves \eqref{eqn:4.32}. Thus it completes the proof of the claim.

\end{proof}

\begin{proof}[Proof of Theorem \ref{mainhigh}.]
The $Q$-curvature in higher dimensions has complicated expression. By its formal definition, the integral of the $Q$-curvature over a closed manifold is invariant under conformal change of the metric. By  results of S.  Alexakis's  \cite{Alex1, Alex2} on classfication of global conformal invariants on even dimensional manifolds, the $Q$-curvature is a linear combination a local conformal invariant $W(g)$, a divergence term and the Pfaffian of the curvature. More precisely,
\begin{equation*}
Q(g)=W(g)+div_i T^i(g)+A \cdot \mathrm{Pfaff(\Omega)},
\end{equation*}
where $W(g)$ is a local conformal invariant of weight $-n$, $T^i(g)$ is an intrisic vector field of weight $- n + 1$,
$A=2^{\frac{n}{2}-2}(\frac{n-2}{2})!$.

By a classical result which is essentially due to Weyl \cite{Weyl},
an intrinsic vector field $T^i(g)$ is a linear combination
$$T^i(g)= \sum_{q\in Q} a_q C^{q,i}(g).$$
Each $C^{q,i}(g)$ is a partial contraction with one free index that takes the form
$$C^{q,i}(g)=pcontr(\nabla^{(m_1)}_{r_1\dots r_{m_1}}R_{i_1j_1k_1l_1} \otimes \cdots \otimes \nabla^{(m_a)}_{t_1\dots t_{m_a}}R_{i_aj_ak_al_a})$$
with $\sum_{t=1}^a (m_t+2)=n-1$.

Apparently, on locally conformally flat manifolds, the local conformally invariant $W(g)$ vanishes. The Pfaffian of curvature $\mathrm{Pfaff(\Omega)}$, by Lemma \ref {Pfaffian high} is equal to
$(n -1)!!\cdot\det (d\vec n)$. By Lemma \ref {vanishing2}, the divergence term of weight $- n + 1$ also vanishes:
\begin{equation*}\int_ M div_i T^i(g)dv_g = 0. \end{equation*}
Therefore,
\begin{equation}
\begin{split}
\int_ M Q_g dv_g =& A \int_ M 	\mathrm{Pfaff(\Omega)} dv_g\\
=&A (n-1)!! \int_ M \det (d\vec n)dv_g\\
\end{split}
\end{equation}
By Lemma \ref{det}, this is equal to
$$A (n-1)!! |\mathbb{S}^n| m.
$$
Since $|\mathbb{S}^n|= \frac{2^{\frac{n}{2}+1} \pi^{\frac{n}{2}}}{(n-1)!!}$
\begin{equation}
\begin{split}
&A (n-1)!! |\mathbb{S}^n| m= 2c_n m,\\
\end{split}
\end{equation}
where $c_n= 2 ^ {n -2} (\frac {n -2} 2)! \pi ^ {\frac n 2}$. It is equal to the integral of the $Q$-curvature on the standard $n $-hemisphere $\mathbb{S}^n_+$.
This completes the proof of the theorem.
\end{proof}


\hide{\section{Continuous extension}
\begin{theorem}\label{cont extension}
Under the same assumptions as in Theorem \ref{integer}, and in addition let us assume $(M,g)$ is nonpositively curved. Then $(M,g)$
is properly immthe jobersed and the Gauss map extends continuously to the end of $M$. Namely, the Gauss map extends continuously to $\bar{M}$.
\end{theorem}

\section{Example}
For $w=\log\log |x|$, what is the integral of $Q$-curvature, and what is the PDE it satisfies? Is it a normal metric?

}
\hide{
S. Chern
&
R
.
Osserman
,
Complete
minimal
surfaces
in
euclidean
n-space,
J
.
Analys
e
Math
.
19(1967)15-34
.}


\begin{bibdiv}
\begin{biblist}

\bib{Alex1}{article}{
AUTHOR = {Alexakis, Spyros},
TITLE = {On the decomposition of global conformal invariants. {I}},
JOURNAL = {Ann. of Math. (2)},
FJOURNAL = {Annals of Mathematics. Second Series},
VOLUME = {170},
YEAR = {2009},
NUMBER = {3},
PAGES = {1241--1306},
ISSN = {0003-486X},
CODEN = {ANMAAH},
MRCLASS = {53C20 (53A30)},
MRNUMBER = {2600873 (2011g:53063)},
}

\bib{Alex2}{book}{
author={Alexakis, Spyros},
title={The decomposition of global conformal invariants},
series={Annals of Mathematics Studies},
volume={182},
publisher={Princeton University Press, Princeton, NJ},
date={2012},
pages={x+449},
}

\bib{Branson}{article}{
AUTHOR = {Branson, Thomas P.},
TITLE = {Sharp inequalities, the functional determinant, and the
complementary series},
JOURNAL = {Trans. Amer. Math. Soc.},
FJOURNAL = {Transactions of the American Mathematical Society},
VOLUME = {347},
YEAR = {1995},
NUMBER = {10},
PAGES = {3671--3742},
ISSN = {0002-9947},
CODEN = {TAMTAM},
MRCLASS = {58G26 (22E46 53A30)},
MRNUMBER = {1316845 (96e:58162)},
MRREVIEWER = {Friedbert Pr{\"u}fer},
URL = {http://dx.doi.org/10.2307/2155203},
}



\bib{CQY1}{article}{
AUTHOR = {Chang, Sun-Yung Alice} AUTHOR =  {Qing, Jie} AUTHOR = { Yang, Paul}
TITLE = {On the {C}hern-{G}auss-{B}onnet integral for conformal metrics
on {$\bold R^4$}},
JOURNAL = {Duke Math. J.},
FJOURNAL = {Duke Mathematical Journal},
VOLUME = {103},
YEAR = {2000},
NUMBER = {3},
PAGES = {523--544},
ISSN = {0012-7094},
CODEN = {DUMJAO},
MRCLASS = {53C65},
MRNUMBER = {1763657 (2001d:53083)},
MRREVIEWER = {John Urbas},
}



\bib{CQY2}{article}{
AUTHOR = {Chang, Sun-Yung Alice} AUTHOR =  {Qing, Jie} AUTHOR = {Yang, Paul}
TITLE = {Compactification of a class of conformally flat 4-manifold},
JOURNAL = {Invent. Math.},
FJOURNAL = {Inventiones Mathematicae},
VOLUME = {142},
YEAR = {2000},
NUMBER = {1},
PAGES = {65--93},
ISSN = {0020-9910},
CODEN = {INVMBH},
MRCLASS = {53C21 (58J60)},
MRNUMBER = {1784799 (2001m:53061)},
MRREVIEWER = {Robert McOwen},
}

\bib{Chern}{book}{
author={Chern, S. S.},
author={Chen, W. H.},
author={Lam, K. S.},
title={Lectures on differential geometry},
series={Series on University Mathematics},
volume={1},
publisher={World Scientific Publishing Co., Inc., River Edge, NJ},
date={1999},
pages={x+356},
}

\bib{Chern-Osserman}{article}{
author={Chern, Shiing-Shen},
author={Osserman, Robert},
title={Complete minimal surfaces in euclidean $n$-space},
journal={J. Analyse Math.},
volume={19},
date={1967},
pages={15--34},
}

\bib{Cohn-Vossen}{article}{
AUTHOR = {Cohn-Vossen, Stefan},
TITLE = {K\"urzeste {W}ege und {T}otalkr\"ummung auf {F}l\"achen},
JOURNAL = {Compositio Math.},
FJOURNAL = {Compositio Mathematica},
VOLUME = {2},
YEAR = {1935},
PAGES = {69--133},
ISSN = {0010-437X},
CODEN = {CMPMAF},
MRCLASS = {Contributed Item},
MRNUMBER = {1556908},
}


\bib{Saloff}{article}
{AUTHOR = {Coulhon, Thierry} AUTHOR={Saloff-Coste, Laurent},
TITLE = {Isop\'erim\'etrie pour les groupes et les vari\'et\'es},
JOURNAL = {Rev. Mat. Iberoamericana},
FJOURNAL = {Revista Matem\'atica Iberoamericana},
VOLUME = {9},
YEAR = {1993},
NUMBER = {2},
PAGES = {293--314},
ISSN = {0213-2230},
MRCLASS = {58G99},
MRNUMBER = {1232845 (94g:58263)},
MRREVIEWER = {Robert Brooks},
URL = {http://dx.doi.org/10.4171/RMI/138},
}



\bib{DS1}{article}{
AUTHOR = {David, Guy} AUTHOR ={ Semmes, Stephen},
TITLE = {Strong {$A_\infty$} weights, {S}obolev inequalities and
quasiconformal mappings},
BOOKTITLE = {Analysis and partial differential equations},
SERIES = {Lecture Notes in Pure and Appl. Math.},
VOLUME = {122},
PAGES = {101--111},
PUBLISHER = {Dekker},
ADDRESS = {New York},
YEAR = {1990},
MRCLASS = {30C65 (42B20)},
MRNUMBER = {1044784 (91c:30037)},
MRREVIEWER = {Michel Zinsmeister},
}

\bib{FeffermanGraham}{article}{
AUTHOR = {Fefferman, Charles} AUTHOR={Graham, C. Robin},
TITLE = {The ambient metric},
SERIES = {Annals of Mathematics Studies},
VOLUME = {178},
PUBLISHER = {Princeton University Press},
ADDRESS = {Princeton, NJ},
YEAR = {2012},
PAGES = {x+113},
ISBN = {978-0-691-15313-1},
MRCLASS = {53A30 (53A55 53C20)},
MRNUMBER = {2858236},
MRREVIEWER = {Michael G. Eastwood},
}

\bib{Fiala}{article}{
AUTHOR = {Fiala, F.},
TITLE = {Le probl\`eme des isop\'erim\`etres sur les surfaces ouvertes
\`a courbure positive},
JOURNAL = {Comment. Math. Helv.},
FJOURNAL = {Commentarii Mathematici Helvetici},
VOLUME = {13},
YEAR = {1941},
PAGES = {293--346},
ISSN = {0010-2571},
MRCLASS = {52.0X},
MRREVIEWER = {J. J. Stoker},
}



\bib{Finn}{article}{
AUTHOR = {Finn, Robert},
TITLE = {On a class of conformal metrics, with application to
differential geometry in the large},
JOURNAL = {Comment. Math. Helv.},
FJOURNAL = {Commentarii Mathematici Helvetici},
VOLUME = {40},
YEAR = {1965},
PAGES = {1--30},
ISSN = {0010-2571},
MRCLASS = {53.25},
MRREVIEWER = {T. Klotz},
}


\bib{hartman}{article}{
   author={P. Hartman},
   title={Geodesic parallel coordinates in the large},
   journal={Amer. J. Math.},
   volume={86},
   date={1964},
   pages={705--727},
   issn={0002-9327},
}




\bib{Huber}{article}{
AUTHOR = {Huber, Alfred},
TITLE = {On subharmonic functions and differential geometry in the
large},
JOURNAL = {Comment. Math. Helv.},
FJOURNAL = {Commentarii Mathematici Helvetici},
VOLUME = {32},
YEAR = {1957},
PAGES = {13--72},
ISSN = {0010-2571},
MRCLASS = {30.00 (31.00)},
}

\bib{lu}{article}{
   author={Lu, Zhiqin},
   title={On the lower order terms of the asymptotic expansion of
   Tian-Yau-Zelditch},
   journal={Amer. J. Math.},
   volume={122},
   date={2000},
   number={2},
   pages={235--273},
   issn={0002-9327},
}


\bib{Stein}{book}{
AUTHOR = {Stein, Elias M.},
TITLE = {Harmonic analysis: real-variable methods, orthogonality, and
oscillatory integrals},
SERIES = {Princeton Mathematical Series},
VOLUME = {43},
NOTE = {With the assistance of Timothy S. Murphy,
Monographs in Harmonic Analysis, III},
PUBLISHER = {Princeton University Press},
ADDRESS = {Princeton, NJ},
YEAR = {1993},
PAGES = {xiv+695},
ISBN = {0-691-03216-5},
MRCLASS = {42-02 (35Sxx 43-02 47G30)},
MRREVIEWER = {Michael Cowling}
}

\bib{Varopoulos}{article}{
AUTHOR = {Varopoulos, N. Th.},
TITLE = {Small time {G}aussian estimates of heat diffusion kernels.
{I}. {T}he semigroup technique},
JOURNAL = {Bull. Sci. Math.},
FJOURNAL = {Bulletin des Sciences Math\'ematiques},
VOLUME = {113},
YEAR = {1989},
NUMBER = {3},
PAGES = {253--277},
ISSN = {0007-4497},
CODEN = {BSMQA9},
MRCLASS = {58G11 (35K05 47D05 60J35)},
MRREVIEWER = {Ana Bela Cruzeiro},
}

\bib{YW1}{article}{
AUTHOR = {Wang, Yi},
TITLE = {The isoperimetric inequality and quasiconformal maps on
manifolds with finite total {$Q$}-curvature},
JOURNAL = {Int. Math. Res. Not. IMRN},
FJOURNAL = {International Mathematics Research Notices. IMRN},
YEAR = {2012},
NUMBER = {2},
PAGES = {394--422},
ISSN = {1073-7928},
MRCLASS = {53C21 (53C20)},
MRNUMBER = {2876387},
MRREVIEWER = {Joseph E. Borzellino},
}

\bib{YW2}{article}{
author={Wang, Yi},
title={The isoperimetric inequality and $Q$-curvature},
journal={Adv. Math.},
volume={281},
date={2015},
pages={823--844},
}

\bib{Weyl}{book}{
author={Weyl, Hermann},
title={The Classical Groups. Their Invariants and Representations},
publisher={Princeton University Press, Princeton, N.J.},
date={1939},
pages={xii+302},
}


\bib{White}{article}{
author={White, Brian},
title={Complete surfaces of finite total curvature},
journal={J. Differential Geom.},
volume={26},
date={1987},
number={2},
pages={315--326},
issn={0022-040X},
}


\bib{WhiteErratum}{article}{
   author={White, Brian},
   title={Correction to: ``Complete surfaces of finite total curvature'' [J.
   Differential Geom.\ {\bf 26} (1987), no.\ 2, 315--326; MR0906393
   (88m:53020)]},
   journal={J. Differential Geom.},
   volume={28},
   date={1988},
   number={2},
   pages={359--360},
   issn={0022-040X},
}
		


\end{biblist}
\end{bibdiv}
\end{document}